\documentclass[a4paper]{amsart}
\usepackage{amssymb,amsmath,float,amsfonts,latexsym,amsthm,comment,pinlabel,comment}
\usepackage{tikz}
\usetikzlibrary{arrows}
 \usepackage{tikz-cd}
\usetikzlibrary{matrix,arrows,decorations.pathmorphing}
\usetikzlibrary{decorations.markings}

\numberwithin{equation}{section}

\usepackage[mathscr]{eucal}

\usepackage[T1]{fontenc}
\usepackage{txfonts}
\usepackage[pdftex,bookmarks=true]{hyperref}
\usepackage{graphicx}
\usepackage{import}
\usepackage{epsfig}
\usepackage{enumerate}
\usepackage[all]{xy}
\theoremstyle{plain}
\newtheorem{thm}[subsection]{Theorem}

\newtheorem{defn}[subsection]{Definition}

\newtheorem{lemma}[subsection]{Lemma}

\theoremstyle{definition}

\newtheorem{notn}[subsection]{Notation}
\newtheorem{rmk}[subsection]{Remark}

\numberwithin{equation}{section}

\author{Subash Chandra Behera}
\address{
School of Mathematics \& Computer Science\\
Indian Institute of Technology Goa\\
At Goa College of Engineering Campus\\
Farmagudi, Ponda-403401 \\
Goa, India}
\email{subash20232102@iitgoa.ac.in}
\author{Shiv Parsad}
\address{
School of Mathematics \& Computer Science\\
Indian Institute of Technology Goa\\
At Goa College of Engineering Campus\\
Farmagudi, Ponda-403401 \\
Goa, India} 
\email{shiv@iitgoa.ac.in}

\begin{document}

\title{Angle Parametrization of Teichm\"uller space and hyperelliptic surfaces}



\maketitle

\begin{abstract}
Let $S_g$ be a closed orientable surface of genus $g \geq 2$, and let $\mathcal{T}_g$ be the Teichmüller space of $S_g$. Let $\mathcal{H}_g$ denotes the space of all hyperelliptic surfaces of genus $g$.  For $g\geq 3$, we have proved that $\mathcal{T}_g$ can be parametrized by $6g-5$ angle parameters. We also prove that for $g\geq 2$, $\mathcal{H}_g$ can be parametrized by $4g-2$ angle  parameters.
\end{abstract}

\maketitle


\section{Introduction}
Let $S_g$ be a closed orientable surface of genus $g \geq 3$, and $\mathcal{T}_g$ be the Teichmüller space of $S_g$. Each element in $\mathcal{T}_g$ can be realized as a marked hyperbolic polygon, a concept initially explored by Coldway and Zieschang; see \cite{HZ70, HZ80}. In these works, it is also proved that $\mathcal{T}_g$ is homeomorphic to $\mathbb{R}^{6g-6}$. Alternative proofs can be found in \cite{PB, PSS}.

In \cite{PSS}, Schaller introduced a parametrization of $\mathcal{T}_g$ by considering lengths as parameters. This approach involves the study of a hyperbolic polygon in which opposite sides are identified based on specific angle conditions. This polygon is referred to as the canonical polygon. The Teichmüller space $\mathcal{T}_g$ can be identified with the space of all canonical polygons (up to some equivalence). Specifically, Schaller focused on a $4g$-sided canonical polygon and chose $6g-5$ length parameters for its parametrization. In \cite{UH},  Hamenstädt constructs $6g-5$ simple closed curves to give an embedding of $\mathcal{T}_g$ into $\mathbb{R}^{6g-5}$. In \cite{OY2}, Okumara has obtained explicit global real analytic parametrizations by angle parameters for $\mathcal{T}_2$. This article introduces a novel parametrization of $\mathcal{T}_g$ using angles for $g\geq 3$. We have proved the following result:

\begin{thm}\label{thm:1}
  For $g\geq 3$, $\mathcal{T}_g$ can be parametrized by $6g-5$ angle parameters.
\end{thm}

We explore the elements within $\mathcal{T}_g$ where the surface is hyperelliptic, denoting the space of all such elements as $\mathcal{H}_g$ ( for more details, we refer the reader to \cite{FK, NR, MR} ). In \cite{PSS}, Schaller proved that a closed hyperbolic surface is hyperelliptic if and only if the surface possesses a corresponding canonical polygon with equal opposite angles. It is a fact that all closed hyperbolic surfaces of genus $2$ are hyperelliptic \cite{PSS1, PSS}. We study the space $\mathcal{H}_g$ using geodesic angles. In particular, we prove the following:

\begin{thm}\label{thm:2}
  For $g\geq 2,$  $\mathcal{H}_g$ can be parametrized by $4g-2$ angle  parameters.
\end{thm}

It follows from a result of Harvey~\cite{HW, HM} that  $\mathcal{H}_g$ has dimension $4g-2$ and so, the angle parametrization given in Theorem \ref{thm:2} is optimal. The parametrization through angles proves to be an efficient method for understanding the deformation of hyperbolic surfaces, offering an alternative perspective compared to length parameters. The proof of this result uses the study of canonical polygons and basic hyperbolic trigonometry.


\section{Fuchsian groups, Canonical polygons and Teichm\"uller space}
This section begins with the basics of hyperbolic geometry. Further we define the notion of Fuchsian groups and hyperbolic surface. We also define the notion of canonical polygon which we used to define the notion of Teichm\"uller space. The most of the material in this section can be found in \cite{PSS} (For more details, see \cite{SK,PB}).

Let $\mathbb{H}$ denote the upper halfplane and $PSL(2,\mathbb{R})$ be the group of orientation preserving isometries of $\mathbb{H}$.
Equivalently,
\[PSL(2, \mathbb{R}) = \left\{ 
\begin{pmatrix}
a & b \\
c & d
\end{pmatrix}
\ \middle|\ a, b, c, d \in \mathbb{R} \textit{ and } ad - bc = 1
\right\} / \{\pm I\}
\]
$ PSL(2,\mathbb{R})$ acts on $\mathbb{H}$ by m\"obius transfermation as follow : $\displaystyle \begin{pmatrix}
a & b \\
c & d
\end{pmatrix}z = \frac{az+b}{cz+d}.$

\subsection{Classifications of Isometries}
Each $T\in PSL(2,\mathbb{R})$ can be represented by a pair of matrices $\pm A\in SL(2, \mathbb{R})$. Define $tr(T)=|{tr(A)}|$. $T$ is said to be elliptic if $tr(T)< 2$, parabolic if $tr(T)=2$, and hyperbolic if $tr(T)>2$.

\begin{defn}
    A Fuschian group is a discrete subgroup of $PSL(2,\mathbb{R})$.
\end{defn}

\begin{thm}
    Let $\Gamma$ be a Fuschian group without elliptic elements, then $\mathbb{H}/\Gamma$ is a hyperbolic surface.
\end{thm}
\begin{defn}
    Let $\Gamma$ be a Fuschian group. A fundamental domain for $\Gamma$ is a measurable subset $D$ of $\mathbb{H}$ such that
    \begin{enumerate}[(i)]
        \item $\bigcup_{\gamma \in  \Gamma} \gamma(D)=\mathbb{H}$, and
        \item $int(\Bar{D})\bigcap int(\gamma(\Bar{D}))=\phi$, for $\gamma\neq identity$.
    \end{enumerate}
\end{defn}
\begin{defn}\label{canonical polygon}\cite{PSS}
    Let $g \geq 2$ be an integer. A canonical polygon $P(g)$ is a polygon with $4 g$ sides, denoted by $a_1, \ldots, a_{4 g}$, ordered clockwise, and angles $\alpha_i$ between $a_i$ and $a_{i+1}, i=1, \ldots, 4 g$ (indices are taken modulo $4 g$), such that
\begin{enumerate}[(i)]
\item $a_i$ and $a_{i+2 g}$ have the same length, $i=1, \ldots, 2 g$;

\item the sum of the angles of $P(g)$ is $2 \pi$;

\item $0<\alpha_i<\pi, i=1, \ldots, 4 g$;

\item $\alpha_1=\alpha_{2 g+1}$;

\item $\displaystyle \sum_{i=1}^g \alpha_{2 i-1}+\sum_{i=g+1}^{2 g} \alpha_{2 i}=\sum_{i=1}^g \alpha_{2 i}+\sum_{i=g+1}^{2 g} \alpha_{2 i-1}$.

\end{enumerate}
\end{defn}
\begin{figure}[H]
    \centering
    \includegraphics[width=4cm]{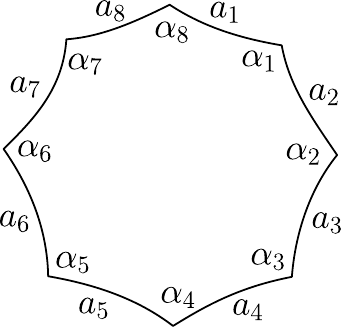}
    \caption{A canonical polygon $P(2).$}
    
\end{figure}

\begin{rmk}
    It can be seen from the local picture ( see Figure \ref{local} ) that the curves $a_{i}'s$ intersect transversally at a point and condition (iv) and (v) of the definition of canonical polygon ensures that $a_1$ and $a_2$ projects to geodesics in the hyperbolic surface.
\begin{figure}
    \centering
    \includegraphics[width=5cm]{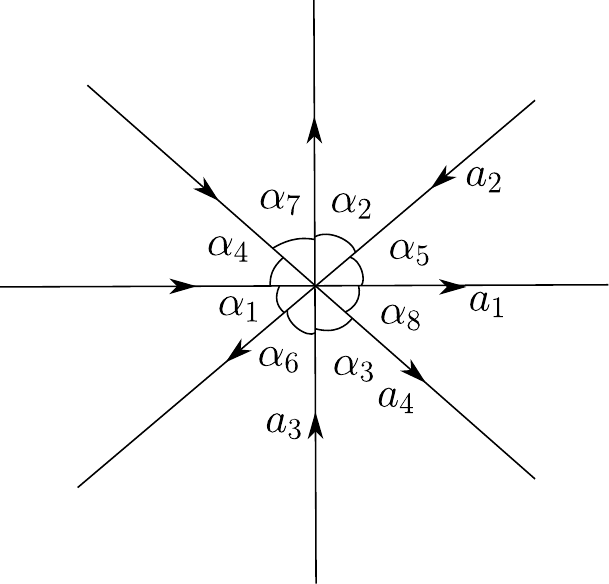}
    \caption{}
    \label{local}
\end{figure}
\end{rmk}
    
\begin{thm}(Poincaré)\label{PT}
     A canonical polygon $P=P(g)$ is the fundamental domain of a Fuchsian group $\Gamma$, and $\mathbb{H} / \Gamma$ is a closed hyperbolic surface of genus $g$. The group $\Gamma$ is generated by the $2g$ elements $\gamma_i$ where $\gamma_i$ is defined by the conditions $\gamma_i(P) \cap \operatorname{int}(P)=\varnothing$ and $\gamma_i\left(a_i\right)=a_{i+2 g}$ if $i$ is odd and $\gamma_i\left(a_{i+2 g}\right)=a_i$ if $i$ is even, $i=1, \ldots, 2 g$.
\end{thm}
For proof, see \cite{CLS,BM,HP, AFB}.

\begin{rmk}
    If $M = \mathbb{H}/\Gamma$ is a hyperbolic surface, then $P(g)$ is the fundamental domain, with the choice $\gamma_1\gamma_2\ldots\gamma_{2g}\gamma_1^{-1}\gamma_2^{-1}\ldots\gamma_{2g}^{-1}=id$.
\end{rmk}

\subsection{Teichm\"uller Space}\label{equiv}
In this section, we define the Teichmüller space using canonical polygons. Let $\mathcal{P}(g)$  denote the space of all canonical polygons. For $P(g),P^{\prime}(g)\in \mathcal{P}(g)$,  Two canonical polygons $P(g)$ and $P^{\prime}(g)$ are said to be equivalent if and only if there is an isometry mapping the side $a_i(P(g))$ to the side $a_i\left(P^{\prime}(g)\right), i=1, \ldots, 4 g$. It follows from the Theorem \ref{PT} that each point in $\mathcal{P}(g)$ defines a hyperbolic surface. In this context, we define:
\begin{defn}
    The Teichm\"uller space $\mathcal{T}_g$ is defined to be the set of equivalence classes of $\mathcal{P}(g)$ with the topology $P_j(g) \rightarrow P(g)$  in $\mathcal{P}(g)$ if and only if the lengths of all sides converge and all angles converge, more precisely, if and only if
$$
a_i\left(P_j(g)\right) \rightarrow a_i(P(g)), \quad i=1, \ldots, 4 g,
$$
where, $a_i\left(P_j(g)\right )$ is the side $a_i$ of $P_j(g)$ and
$$
\alpha_i\left(P_j(g)\right) \rightarrow \alpha_i(P(g)), \quad i=1, \ldots, 4 g,
$$
where, $\alpha_i\left(P_j(g)\right)$ is the angle $\alpha_i$ of $P_j(g)$ .
\end{defn}


\section{Hyperbolic Trigonometry}\label{sec:3}
In this section, we state the basic hyperbolic trigonometry formulae and develop some basic lemmas that are needed to prove our main result.
\begin{lemma}\label{sine and cosine rules}\cite{PB, JA}
 
Let $T$ be a triangle with angles $\alpha, \beta, \gamma$ and sides of lengths $a, b, c$. (See Figure \ref{Triangle}) Then

\begin{figure}[H]
    \centering
    \includegraphics[width=4cm]{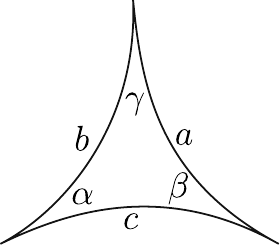}
    \caption{Triangle.}
    \label{Triangle}
\end{figure}
 \begin{enumerate}[(i)]

 \item $\displaystyle \frac{\sinh a}{\sin \alpha}= \frac{\sinh b}{\sin \beta}=\frac{\sinh c}{\sin \gamma}$;
 
     \item  $\cos \gamma=-\cos \alpha \cos \beta+\sin \alpha \sin \beta \cosh c $;

     \item  $\cosh c = \cosh a \cosh b - \sinh a \sinh b \cos \gamma$.
 \end{enumerate}

\end{lemma}

\begin{notn}\label{res:Triangle sides}
    Let $T$ be the triangle described above, then
\begin{enumerate}[(i)]
\item $\gamma=\psi(\alpha, c, \beta)$, where $\psi(\alpha, c, \beta) = \cos^{-1}\left(-\cos \alpha \cos \beta+\sin \alpha \sin \beta \cosh c\right)$;
   \item $\gamma=f(c, \beta, a)$, where $\displaystyle f(c, \beta, a) = \sin^{-1}\left(\frac{\sinh c \sin \beta}{\sinh a \sqrt{\left(\cosh a \cosh c - \sinh a \sinh c \cos \beta \right)^{2}- 1}}\right)$;
   \item $b=g(\alpha, \beta, \gamma)$, where $ \displaystyle g(\alpha, \beta, \gamma) =\cosh^{-1} \left( \frac{\cos \beta + \cos \alpha \cos \gamma}{\sin \alpha \sin \gamma} \right)$.
\end{enumerate}
\end{notn}

\begin{lemma}\label{alpha}
    Let ABCD be a quadrilateral. Let $AC$ be the diagonal and separate it into two triangles such that $\angle ABC=\alpha, \angle BCA=\beta_1, \angle ACD= \beta_2, \angle CDA= \gamma, \angle DAC=\delta_2, \angle BAC=\delta_1$, then $\alpha=h\left(\beta_1, \beta_2, \gamma, \delta_2, \delta_1\right)$, where
    
    $$h\left(\beta_1, \beta_2, \gamma, \delta_2, \delta_1\right)=\cos^{-1}\left(-\cos \beta_1 \cos \delta_1 + \sin \beta_1 \sin \delta_1 \cosh(g(\beta_2, \gamma, \delta_2))\right)$$
\end{lemma}

\begin{figure}[H]
    \centering
    \includegraphics[width=4cm]{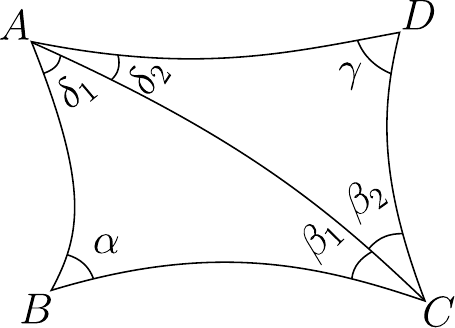}
    \caption{}
    
\end{figure}
\begin{proof}
  By Remark \ref{res:Triangle sides}, we have:
\begin{equation*}
\textit {length of AC}  = g(\beta_2, \gamma, \delta_2)
\end{equation*}

And by Lemma \ref{sine and cosine rules}, we can express $\cos \alpha$ as:
\begin{equation*}
\cos \alpha = -\cos \beta_1 \cos \delta_1 + \sin \beta_1 \sin \delta_1 \cosh b_1
\end{equation*}

This implies:
\begin{equation*}
\cos \alpha = -\cos \beta_1  \cos \delta_1 + \sin \beta_1 \sin \delta_1  \cosh(g(\beta_2, \gamma, \delta_2))
\end{equation*}

Thus, we can calculate $\alpha$ as:
\begin{equation*}
\alpha = \cos^{-1}\left(-\cos \beta_1 \cos \delta_1 + \sin \beta_1 \sin \delta_1 \cosh(g(\beta_2, \gamma, \delta_2))\right)
\end{equation*}

We denote this function as:
\begin{equation*}
h(\beta_1, \beta_2, \gamma, \delta_2, \delta_1) = \cos^{-1}\left(-\cos \beta_1 \cos \delta_1 + \sin \beta_1 \sin \delta_1 \cosh(g(\beta_2, \gamma, \delta_2))\right).
\end{equation*}

\end{proof}
\begin{lemma}\label{res:Schmutz} \cite{PSS}
    Let $T$ be a triangle with the notation of Figure \ref{Triangle}. Let $T^{\prime}$ be a triangle with sides of length $a^{\prime}, b^{\prime}, c^{\prime}$ and angles $\alpha^{\prime}, \beta^{\prime}, \gamma^{\prime}$. Let $a=a^{\prime}$ and $b=b^{\prime}$. Then
    $$ c^{\prime} > c \Longleftrightarrow \gamma^{\prime} > \gamma $$
\end{lemma}


\section{Angle parametrization of Teichm\"uller space}
In this section, we prove Theorem~\ref{thm:1}. We need the following lemma:

\begin{lemma}\label{Quadrilateral lemma}
 
Given two quadrilaterals $ABCD$ and $A'B'C'D'$ with the same lengths of four sides, and denoting their respective angles as $\alpha, \beta, \gamma, \delta$ and $\alpha', \beta', \gamma', \delta'$ in the natural order (see Figure \ref{qd}).
If $\alpha+\beta+\gamma+\delta = \alpha'+\beta'+\gamma'+\delta'$ and $\alpha - \beta + \gamma - \delta = \alpha' - \beta' + \gamma' - \delta'$, then  $(\alpha', \beta', \gamma', \delta') = (\alpha, \beta, \gamma, \delta)$.
\end{lemma}
\begin{figure}[H]
    \centering
    \includegraphics[width=10cm]{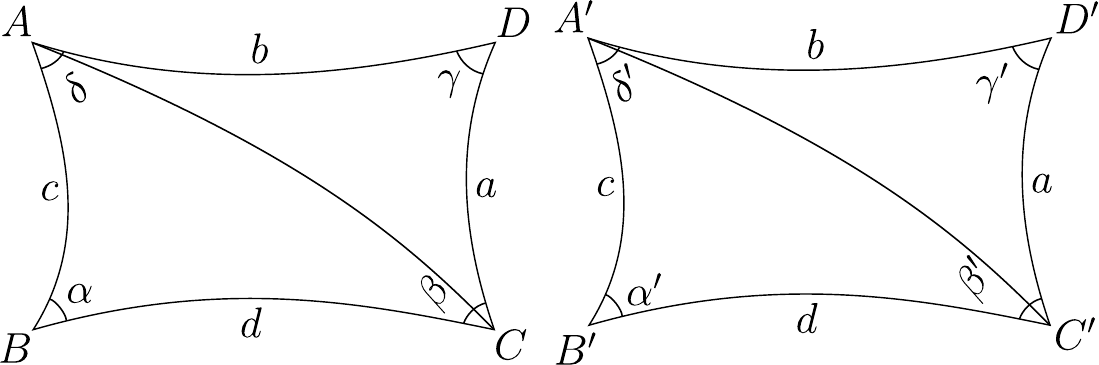}
    \caption{}
    \label{qd}
\end{figure}

\begin{proof}
    If possible, $(\alpha', \beta', \gamma', \delta') \neq (\alpha, \beta, \gamma, \delta)$. Without loss of generality assume that $\alpha > \alpha'$. Let $AC$ and $A'C'$ represent the respective diagonals of quadrilaterals $ABCD$ and $A'B'C'D'$.

 By Lemma \ref{res:Schmutz}, it follows that the length of $AC$ is greater than  the length of $ A'C'$, implying $\gamma > \gamma'$. Since $\alpha + \beta + \gamma + \delta = \alpha' + \beta' + \gamma' + \delta'$, we observe that $\beta + \delta < \beta' + \delta'$, which implies $\alpha - \beta + \gamma - \delta > \alpha' - \beta' + \gamma' - \delta'$. This, however, leads to a contradiction.
\end{proof}

\begin{proof}[The proof of Theorem~\ref{thm:1}]

Let $g \geq 3$ and $P(g)$ be a cononical polygon with sides $a_i$ and angle $\alpha_i $ between $a_i$ and $a_{i+1}$ for $i=1,...,4g$ (the indices are taken modulo $4g$). Let $Q_i= a_i\cap a_{i+1}$ and $b_i$ be the geodesic segment between $Q_{4g}$ and $Q_{i+1}$ for $i=1,\dots, 4g-4, i\neq 2g$. $b_{2g}$ is defined to be the geodesic segment joining $Q_{2g}$ and $Q_{2g+2}$. $\phi_i$ be the angle between the segment $b_i$ and $a_{i+1}$, for $i=1,..,4g-5, i\neq 2g$. $\beta_i$ be the angle between $b_{i-1}$ and $b_i$ for $i=2,..., 4g-5, i\neq \beta_{2g}, \beta_{2g+1}$. Let $\beta_1$ be the angle between $a_1$ and $b_1$. Denote $\gamma$ to be the angle between the segment $b_{2g-1}$ and $b_{2g+1}$.
$P(g)$ is separated by the geodesic segments $b_1, b_2, \dots,  b_{4g-5}$ into one quadrilateral $S$ and $4g-4$ triangles $T_i, \;i= 1,\dots,4g-4$. Further $T_i$ has sides $b_{i-1}, b_i, a_{i+1}$, for $i=2,\dots,4g-4, i \neq 2g, 2g+1$, the triangle $T_1$ has sides $a_1, a_2, b_1$; the triangle $T_{2g}$ has sides $b_{2g}, a_{2g+1}, a_{2g+2}$ and the triangle $T_{2g+1}$ has sides $b_{2g-1}, b_{2g}, b_{2g+1}$ ( see Figure \ref{Fig1}; this idea of decomposition into triangles and quadrilateral is used by Schaller \cite{PSS} ). Let $L(a_i), L(b_j)$ denote the length of $a_i, b_j$ respectively for $1 \leq i \leq 4g$ and $1\leq j \leq 4g-4$.

\begin{figure}[H]
    \includegraphics[width=8cm]{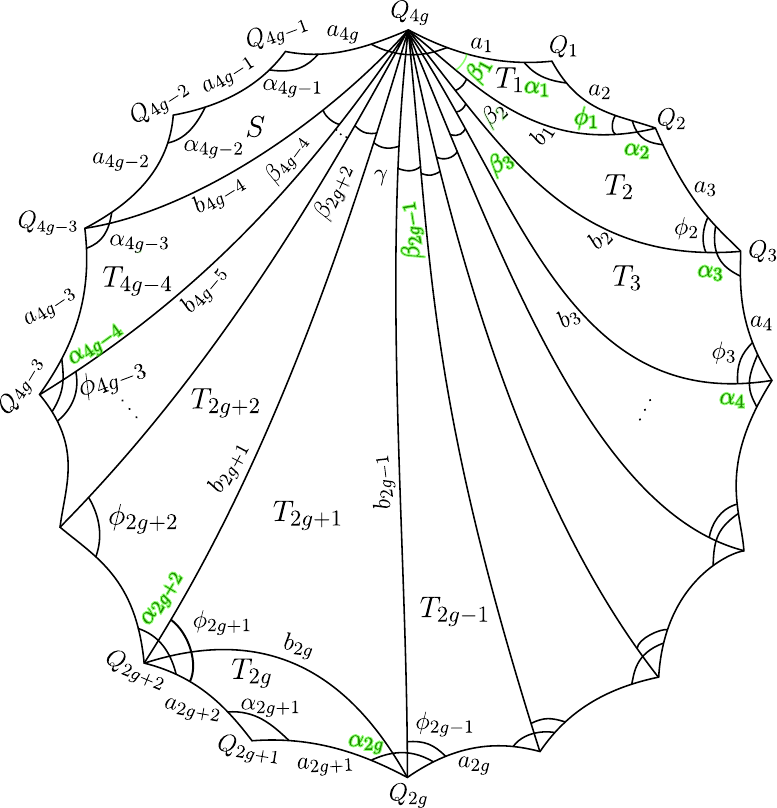}
    \caption{Canonical 4g-gon, here coloured angles are the parameters.}  
    \label{Fig1}
\end{figure}
For $\theta=\left(\theta_1,...,\theta_{6g-5}\right) \in \mathbb{R}^{6g-5}$,
$\theta_i=\alpha_i$, for $i=1,...,4g-4, i\neq 2g+1, \theta_{2g+1}=\phi_1$, and $\theta_{4g-4+i}=\beta_i$, for each $i=1,\dots,2g-1$. Let $f,g,h$, and $\psi$ be defined as in Section \ref{sec:3}.
We now compute each and every $\phi_i$ in terms of $\theta_i's$.

$\phi_2=h\left(\alpha_2-\phi_1, \phi_1, \alpha_1, \beta_1, \beta_2\right)=h\left(\theta_2-\theta_{2g+1},\theta_{2g+1} , \theta_1, \theta_{4g-3}, \theta_{4g-2}\right)$,

$\phi_i = h(\alpha_i - \phi_{i-1}, \phi_{i-1}, \alpha_{i-1} - \phi_{i-2}, \beta_{i-1}, \beta_i) = h(\theta_i - \phi_{i-1}, \phi_{i-1}, \theta_{i-1} - \phi_{i-2}, \theta_{4g-5+i}, \theta_{4g-4+i})$ for \(i = 3, \ldots, 2g-1\).

It is evident that the triangle $T_{2g}$ and the triangle $T_1$ are congruent, implying that $L(b_{2g}) = L(b_1)$. By iterative use of $f,g,h$ and $\psi$, we have

$L(b_{2g-1})= g \left(\phi_{2g-1}, \alpha_{2g-1}-\phi_{2g-2}, \beta_{2g-1}\right)$,

$L(b_{2g})=L(b_1)= g\left(\phi_1, \alpha_1, \beta_1\right)$,

$\gamma=f\left(L(b_{2g}), \alpha_{2g}-\phi_{2g-1}-\beta_1, L(b_{2g-1})\right)$,

$\phi_{2g+1}=\phi_1+ \psi\left(\alpha_{2g}-\phi_{2g-1}-\beta_1, L(b_{2g-1}), \gamma\right)$, 

$L(b_{2g+1})=g\left(\phi_{2g+1}-\phi_1, \alpha_{2g}-\phi_{2g-1}-\beta_1, \gamma\right)$,

$L(a_{2g+3})=L(a_3)=g\left( \alpha_2-\phi_1, \beta_2, \phi_2 \right)$,

$\beta_{2g+2}=f\left(L(a_{2g+3}), \alpha_{2g+2}-\phi_{2g+1}, L(b_{2g+1})\right)$,

$\phi_{2g+2}=\psi\left(\alpha_{2g+2}-\phi_{2g+1}, L(b_{2g+1}), \beta_{2g+2}\right) $,


$L(b_{2g+i-1})=g\left(\phi_{2g+i-1}, \alpha_{2g+i-1}-\phi_{2g+i-2}, \beta_{2g+i-1}\right)$,

$L(a_{2g+i+1})=L(a_{i+1})=g\left( \alpha_i-\phi_{i-1}, \beta_i, \phi_i \right)$,

$\beta_{2g+i}=f\left(L(a_{2g+i+1}), \alpha_{2g+i}-\phi_{2g+i-1}, L(b_{2g+i-1})\right)$, and

$\phi_{2g+i}=\psi\left(\alpha_{2g+i}-\phi_{2g+i-1}, L(b_{2g+i-1}), \beta_{2g+i}\right)$, for $i=3,\dots2g-3$.

 \vspace{0.1cm}
Thus all $\phi_i's$ are determined in terms of $\theta_i's$ uniquely, so triangles $T_i$'s are also determined uniquely. It remains to show that the quadrilateral $S$ is uniquely determined. Since the sum of the interior angles of the 4g-gon is $2\pi$, $\angle Q_{4g-1}Q_{4g}Q_{4g-3}+ \angle Q_{4g}Q_{4g-3}Q_{4g-2}+\alpha_{4g-2}+\alpha_{4g-1}$ is constant. By condition (v) of definition \ref{canonical polygon}, 
$ \angle Q_{4g-1}Q_{4g}Q_{4g-3}- \angle Q_{4g}Q_{4g-3}Q_{4g-2}+\alpha_{4g-2}-\alpha_{4g-1}$ is constant. Thus, by Lemma \ref{Quadrilateral lemma}, the quadrilateral is uniquely determined.

\end{proof}

\section{parametrization of the space of hyperelliptic surfaces }

\begin{defn}
A hyperelliptic surface is a closed hyperbolic surface of genus $g$ which has an isometry $\phi$, with $\phi^2= id$ and with exactly $2g+2$ fixed points. 

\end{defn}

\begin{thm} \cite{PSS}
Let $M$ be a closed hyperbolic surface of genus $g$.
    M is hyperelliptic if and only if M has a corresponding canonical polygon with equal opposite angles.
\end{thm}

We define $\mathcal{H}_g$ to be the space of all hyperelliptic surfaces. In order to parametrize $\mathcal{H}_g$, we need the following lemma:

\begin{lemma}\label{res: quadrilateral lemma}
Let $ABCD$ be a convex quadrilateral with diagonal $AC$. The lengths of sides $AB$ and $AD$ are given. If the measures of angles $\angle BAC$ and $\angle CAD$ are known, and the sum of angles $\angle ABC + \angle BCD + \angle CDA$ is known, then the quadrilateral $ABCD$ is uniquely determined.

\begin{figure}[H]
    \centering
    \includegraphics[width=4cm]{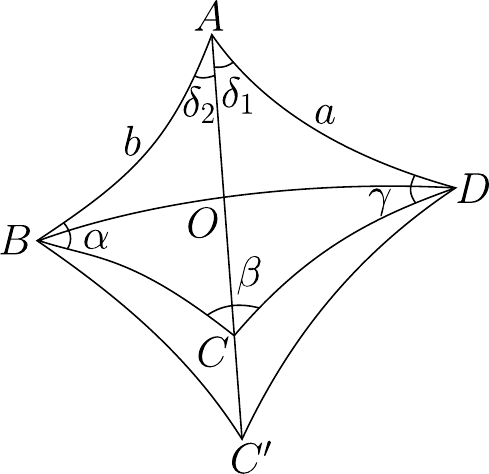}
    \caption{}

\end{figure}
\end{lemma}

\begin{proof}
Let $a, b, t$ be the length of $AD, AB$ and $AC$, respectively. Also, $\delta_1$ and $\delta_2$ denotes the angles $\angle BAO$ and $\angle DAO$, respectively. Let $O$ be the intersection point of diagonal $AC$ and $BD$. Increase $t$ without changing $a$, $b$, $\delta_1$ and $\delta_2$. Let $C'$ be the corresponding point of $C$ after the increment.

Therefore, the area of the quadrilateral $ABCD$ is less than the area of the quadrilateral $ABC'D$, implying that $\alpha+\beta+\gamma$ is a continuously decreasing function of $t$. Thus, for each choice of $\alpha+\beta+\gamma$ between $0$ and $\pi$, there exists a unique quadrilateral.
\end{proof}

\begin{figure}
    \centering
    \includegraphics[width=7cm]{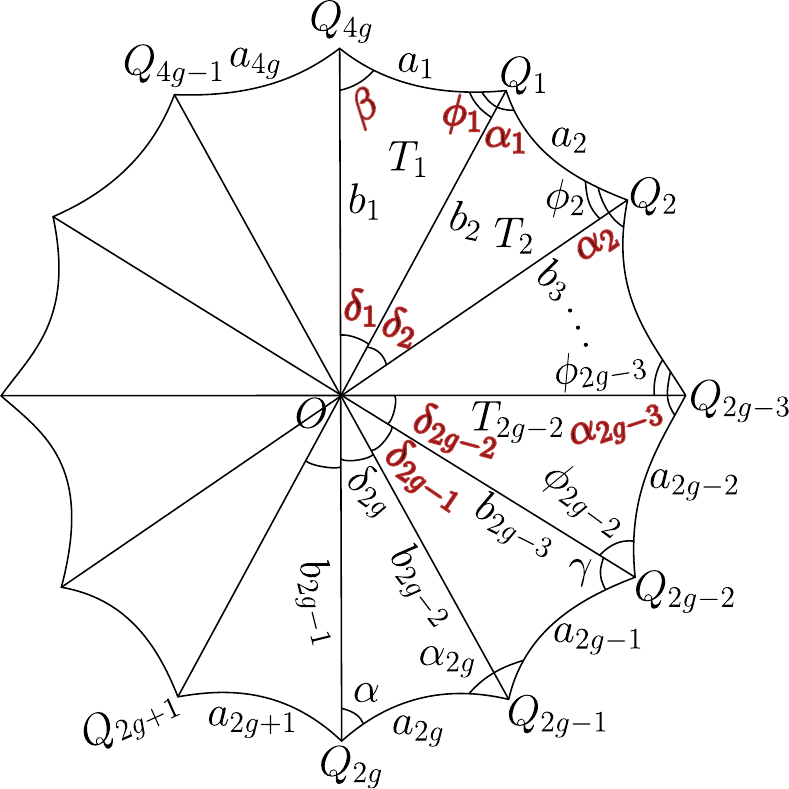}
    \caption{Coloured angles are the parameters}
    \label{Fig2}
\end{figure}

\begin{proof}[The proof of Theorem~\ref{thm:2}]
Consider $g \geq 2$, and let $P(g)\in \mathcal{H}_g$ be a canonical polygon with sides $a_i$ and angles $\alpha_i$ between $a_i$ and $a_{i+1}$ for $i=1, \ldots, 4g$ (indices taken modulo $4g$). 
Define $Q_i$ as the intersection point of $a_i$ and $a_{i+1}$. 
Set $k_i$ as the geodesic segment between $Q_i$ and $Q_{2g+i}$ for $i=1, 2, \ldots, 2g$. One can note that each $k_i$ intersects at a point, which happens to be the midpoint of these $k_i$. 
Let $O$ be the intersection point of all $k_i$.
 Define $b_1 = OQ_{4g}$ and $b_i = OQ_{i+1}$ for $i=2, \ldots, 4g$. 
 Let $\beta = \angle OQ_{4g}Q_1$, $\delta_1 = \angle Q_{4g}OQ_1$, and $\delta_i = \angle Q_{i-1}OQ_{i+1}$ for $i=2, \ldots, 2g$. Define $\phi_1 = \angle Q_{4g}Q_1O$ and $\phi_i = \angle Q_{i-1}Q_iO$ for $i=2, \ldots, 2g$. Let $T_1$ be the triangle with angles $\delta_1, \beta$, and $\phi_1$. $T_i$ be the triangle with angles $\delta_i, \alpha_{i-1}-\phi_i$, and $\phi_i$ for $i=2, \ldots, 2g$ ( see Figure \ref{Fig2} ).

For $\theta = (\theta_1, \ldots, \theta_{4g-2}) \in \mathbb{R}^{4g-2}$, set $\theta_i = \alpha_i$ for $i=1, 2, \ldots, 2g-3$. Set $\theta_{2g-3+i} = \delta_i$ for $i=1, 2, \ldots, 2g-1$. Finally, let $\theta_{4g-3} = \beta$ and $\theta_{4g-2} = \phi_1$. We want to show that all triangles can be uniquely determined by these $4g-2$ parameters. Since hyperelliptic involution is an isometry, it is enough to show the triangles $T_i$ for $i=1, 2, \dots, 2g$ can be determined uniquely.
Let $h$ be defined as in Lemma \ref{alpha}. We get

\[
\begin{aligned}
\phi_2 &= h\left(\alpha_1-\phi_1, \phi_1, \beta, \delta_1, \delta_2\right), \\
 \text{ and } \phi_i &= h\left(\alpha_{i-1}-\phi_{i-1}, \phi_{i-1}, \alpha_{i-2}-\phi_{i-2}, \delta_{i-1}, \delta_i\right), \text{ for } i=1,\dots 2g-2.
\end{aligned}
\]

Thus all triangles $T_i$, for $i=1, 2, ..., 2g-2$ determined uniquely in terms of $\theta_i$'s. Now, it remains to show that the quadrilateral $OQ_{2g}Q_{2g-1}Q_{2g-2}$ is uniquely determined.
 It can be observed that $b_{2g-1} = b_1$. According to Definition \ref{canonical polygon}, $\phi_{2g} + \alpha_{2g} + \gamma = 2\pi - \sum_{i=1}^{2g-3} \alpha_i + \beta - \phi_{2g-2}$ is constant. Additionally, $\delta_{2g} = \pi - \sum_{i=1}^{2g-1} \delta_i$, which is also a constant. Therefore, according to Lemma \ref{res: quadrilateral lemma}, the quadrilateral $OQ_{2g}Q_{2g-1}Q_{2g-2}$ is uniquely determined.

\end{proof}
\subsection{Concluding remarks}
   In the local picture (refer to Figure \ref{localpicture}), when considering a canonical $12$-gon with equal opposite angles, it's evident that all $a_i$'s and $k_i$'s are geodesics within the quotient of the $12$-gon. This observation also holds true for a canonical $4g$-gon with equal opposite angles. Hence the $4g-2$ angles used in the parametrization of $\mathcal{H}_g$ are geodesic angles. It follows from a result of Harvey~\cite{HW,HM} that  $\mathcal{H}_g$ has dimension $4g-2$ and so, the angle parametrization given in Theorem \ref{thm:2} is optimal.

\begin{figure}[H]
    \centering
    \includegraphics[width=11cm]{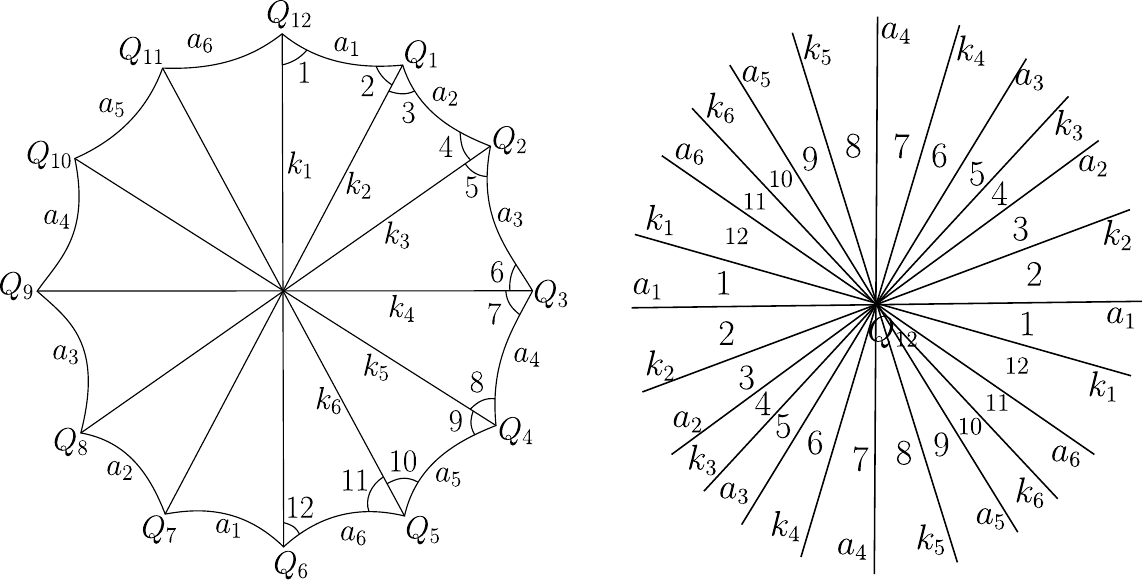}
    \caption{}
\label{localpicture}
\end{figure}




\bibliographystyle{alpha}
\bibliography{angle}

\end{document}